\documentclass{amsart}
\usepackage{amsmath,amssymb,amsthm,mathrsfs}
\usepackage[active]{srcltx}
\usepackage{hyperref}
 \usepackage[pdftex]{graphicx}
 \usepackage{color}

\theoremstyle{plain}
\numberwithin{equation}{section}

\newcommand{\leaveout}[1]{}

\makeatletter
\renewcommand{\p@enumii}{}
\makeatother

\newcommand\C{{\mathbb C}}
\newcommand\D{{\mathbb D}}

\newcommand\Z{{\mathbb Z}}

\newcommand\zplus{{\Z^{+}}}

\newcommand\T{{\mathbb T}}

\newcommand{\dom}[1]{\mathrm{dom}\left(#1\right)}
\newcommand{\bbm}[1]{\begin{bmatrix}#1\end{bmatrix}}

\newcommand{\BC}{{\mathbb C}}\newcommand{\BD}{{\mathbb D}}

\newcommand{\BT}{{\mathbb T}}

\newcommand{\BZ}{{\mathbb Z}}
\newcommand{\cB}{{\mathcal B}}
\newcommand{\cC}{{\mathcal C}}

\newcommand{\cO}{{\mathcal O}}

\newcommand{\cU}{{\mathcal U}}\newcommand{\cV}{{\mathcal V}}
\newcommand{\cW}{{\mathcal W}}\newcommand{\cX}{{\mathcal X}}
\newcommand{\cY}{{\mathcal Y}}
\newcommand{\wtilB}{\widetilde{B}}
\newcommand{\wtilC}{\widetilde{C}}

\newcommand{\wtilQ}{\widetilde{Q}}\newcommand{\wtilR}{\widetilde{R}}

\newcommand{\wtilU}{\widetilde{U}}

\newcommand{\whatQ}{\widehat{Q}}\newcommand{\whatR}{\widehat{R}}

\newcommand{\al}{\alpha}
\newcommand{\be}{\beta}
\newcommand{\ga}{\gamma}\newcommand{\Ga}{\Gamma}
\newcommand{\de}{\delta}

\newcommand{\la}{\lambda}

\newcommand{\si}{\sigma}\newcommand{\Si}{\Sigma}
\newcommand{\im}{\operatorname{Im}}

\newcommand{\kr}{\operatorname{Ker}}

\newcommand{\mat}[1]{\begin{bmatrix} #1 \end{bmatrix}}

\newcommand{\sbm}[1]{\left[\begin{smallmatrix} #1\end{smallmatrix}\right]}
\newcommand{\ov}[1]{{\overline{#1}}}

\newcommand{\tu}[1]{\textup{#1}}
\newcommand{\wtil}[1]{{\widetilde{#1}}}

\newtheorem{lemma}{Lemma}[section]
\newtheorem{theorem}[lemma]{Theorem}
\newtheorem{corollary}[lemma]{Corollary}
\newtheorem{proposition}[lemma]{Proposition}

\theoremstyle{definition}
\newtheorem{example}[lemma]{Example}

\begin{document}

\title[Canonical Wiener-Hopf factorization on the unit circle]{Canonical Wiener-Hopf factorization on the unit circle: matching subspaces versus Riccati equations}

\author[S. ter Horst]{S. ter Horst}
\address{S. ter Horst, Department of Mathematics, Research Focus Area:\ Pure and Applied Analytics, North-West University, Potchefstroom, 2531 South Africa and DSI-NRF Centre of Excellence in Mathematical and Statistical Sciences (CoE-MaSS)}
\email{Sanne.TerHorst@nwu.ac.za}

\author[M. Kurula]{M. Kurula}
\address{M. Kurula, {\AA}bo Akademi Mathematics, Henriksgatan 2, 20500 {\AA}bo, Finland}
\email{Mikael.Kurula@abo.fi}

\author[A.C.M. Ran]{A.C.M. Ran}
\address{A.C.M. Ran, Department of Mathematics, Faculty of Science, VU Amsterdam, De Boelelaan 1111, 1081 HV Amsterdam, The Netherlands and Research Focus: Pure and Applied Analytics, North-West University, Potchefstroom, South Africa}
\email{a.c.m.ran@vu.nl}

\thanks{This work is based on the research supported in part by the National Research Foundation of South Africa (Grant Numbers 127364 and 145688) and the Magnus Ehrnrooth Foundation.}

%47A63(1991–now)Linear operator inequalities
%47A48(1991–now)Operator colligations (= nodes), vessels, linear systems, characteristic functions, realizations, etc.
%47A56(1980–now)Functions whose values are linear operators (operator- and matrix-valued functions, etc., including analytic and meromorphic ones)
%93B28(1980–now)Operator-theoretic methods
%93C05(1973–now)Linear systems in control theory
%93D25(1980–now)Input-output approaches in control theory


\begin{abstract}
Wiener-Hopf factorization is an important tool in the study of block Toeplitz and block Wiener-Hopf operators, and many applications involving these operators. In this paper we compare two approaches to Wiener-Hopf factorization, namely, the more classical approach based on matching invariant subspaces and a more recent approach based on solutions to a non-symmetric Riccati equation. The latter approach is extended to the case of Hilbert space operator-valued functions that are analytic on a neighborhood of the unit disc $\BT$, but need not be rational. In both approaches, existence of canonical right Wiener-Hopf factorization is characterized by existence of a stabilizing solution to a Riccati equation, however, the Riccati equations are not the same. We analyse the solution sets of the Riccati equations and show that they indeed are not the same, but they do have the same stabilizing solution. 
\end{abstract}

\maketitle

%{\it Dedicated to the memory of Harry Dym, in admiration of his contributions to operator theory.}

%\tableofcontents

%%%%%%%%%%%%%%%%%%%%%%%%%%%%%%%%%%%%%%%%%%%%%%%%%%%%%%%%%%%%%%%%%%%%%%%%%%%%%
\section{Introduction}\label{S:Intro}

Let $R(z)$ be a function whose values are bounded linear operators on a given Hilbert space, which is analytic and invertible on an annulus containing the unit circle. A factorization $R(z)=V_-(z)V_+(z)$ is called a \emph{right canonical Wiener-Hopf factorization} when $V_-(z)$ and $V_+(z)$ are also operator-valued functions, on the same Hilbert space, such that $V_+(z)$ is analytic and invertible on the closed unit disc $\ov{\BD}$ with an inverse that is also analytic on $\ov{\BD}$, while  $V_-(z)$ is analytic and invertible on the exterior of the open disc $\BC\backslash \BD$ with an inverse that is also analytic on $\BC\backslash \BD$. When the properties of the factors are reversed, so $R(z)=W_+(z)W_-(z)$, with $W_+(z)$ and $W_-(z)$ corresponding to the interior and the exterior, respectively, of the unit circle, we say that this is a \emph{left canonical Wiener-Hopf factorization}. 

Canonical Wiener-Hopf factorization of matrix- and operator-valued functions has been studied extensively in the last seventy years. It appeared for instance in \cite{GoKr} in connection to applications to systems of Wiener-Hopf integral equations, and in \cite{GF74} in connection to applications to (block) Toeplitz operators. These sources focussed on the matrix-valued case. The technically much more difficult case of operator-valued functions was treated in \cite{GL09}.
For an extensive overview see \cite{GKS}. 

Having some form of a state space realization for the function $R(z)$ turns out to be extremely helpful in determining the Wiener-Hopf factorization. In many sources it is assumed that $R(z)$ has a realization as the transfer function of a continuous time system, with the additional assumption that the value at infinity is an invertible operator (see, e.g., \cite{BGKR08, BGKR10} and the references give there). The latter assumption is weakened in \cite{Coh83}, while in \cite{GK88} $R(z)$ is assumed to be the transfer function of a continuous time differential-algebraic system (i.e., the state equation is a system of differential-algebraic equations). A related type of realization was used in \cite{Gr}.
A rather different representation of $R(z)$, more adapted to Toeplitz operators, was used in \cite{FKR10}; see \eqref{FKRreal} below. While earlier realization formulas for $R(z)$ lead to conditions for canonical Wiener-Hopf factorization in terms of a matching pair of invariant subspaces for two operators, the main result in \cite{FKR10} expresses existence of canonical Wiener-Hopf factorization in terms of existence of a stabilizing solution of a (discrete) algebraic Riccati equation. In this approach the function does not need to be invertible at zero or at infinity, in fact, it does not even need to be defined at zero or at infinity. The disadvantage of the representation of $R(z)$ in \eqref{FKRreal} is that an immediate formula for the inverse of $R(z)$ is lacking. Algebraic Riccati equations were used in earlier sources as well, mainly to discuss factorizations of functions with symmetry, see e.g., \cite{LaRoBook}. In \cite{FKR10} the non-symmetric case is the focus of attention. Recently, in \cite{tHKR25}, we used an approach based on dichotomous systems (see \cite{BGtH18c}) to return to the factorization of operator-valued functions whose values on the unit circle are of the form identity plus a strict contraction. It is known that such functions admit both right and left canonical Wiener-Hopf factorization (see, e.g., \cite{GL09,BGKR10} and the references given there). In \cite{tHKR25}, we provided explicit formulas for the factors in terms of the operators appearing in a realization of the transfer function of a dichotomous system.

The goal of the present paper is twofold. We expand the results of \cite{FKR10} to the infinite dimensional situation in Theorem \ref{thm:FKR10} below, and we shall discuss the connections between the approaches using a matching pair of invariant subspaces as in \cite{BGKR08, BGKR10} and the approach using a non-symmetric algebraic Riccati equation from \cite{FKR10}. For the latter we shall restrict to operator-valued transfer functions of dichotomous systems. 

In Subsection \ref{SubS:StableRep} we expand the result of \cite{FKR10} to the infinite dimensional situation, while the Wiener-Hopf factorization for dichotomous realizations in terms of invariant subspaces is discussed in Subsection \ref{SubS:DichotReal}. The connection between dichotomous realizations and stable representations from \cite{FKR10} is clarified in Subsection \ref{SubS:Connect}. The two approaches described in Subsections \ref{SubS:StableRep} and \ref{SubS:DichotReal} lead to two algebraic Riccati equations, which are strongly related, but may have different sets of solutions. The connection between these two algebraic Riccati equations is analysed in Section \ref{S:Ricc}. Finally, in Section \ref{S:LeftRight} we assume a right Wiener-Hopf factorization exists, and then give a necessary and sufficient condition for existence of a left Wiener-Hopf factorization (compare \cite{BallRan, BGKR10}). 

The following notation and terminology will be used throughout the paper. For Hilbert spaces $\cX$ and $\cY$, $\cB(\cX,\cY)$ denotes the Banach space of bounded linear operators mapping $\cX$ into $\cY$, abbreviated to $\cB(\cX)$ in case $\cX=\cY$. For a subspace $\cV$ of a Hilbert space $\cX$ we write $P_\cV$ for the orthogonal projection onto $\cV$. In the sequel, the term ``operator'' will always mean bounded linear operator, while ``invertible'' means boundedly invertible.  We use $\si(T)$ and $\rho(T)$ for the spectrum and resolvent set of $T$, respectively; also we define
$
	\rho(T)^{-1}:=\left\{\tfrac{1}{\la}\mid 0\neq\la\in\rho(T)\right\}.
$

%%%%%%%%%%%%%%%%%%%%%%%%%%%%%%%%%%%%%%%%%%%%%%%%%%%%%%%%%%%%%%%%%%%%%%%%%%%%%%%%%%%%%%%%%%
\section{Canonical Wiener-Hopf factorization:\ Two approaches}\label{S:WH2approaches}

In this section we discuss two state space approaches to canonical Wiener-Hopf factorization for operator-valued functions, namely the one from \cite{FKR10} based on non-symmetric algebraic Riccati equations (which we extend here to the non-rational Hilbert space operator-valued case) and the one from Chapters 7 and 12 of \cite{BGKR10} based on matching invariant subspaces. Since the method from \cite{FKR10} works for a more general class of functions, we shall discuss that approach first.

\subsection{Canonical Wiener-Hopf factorization in terms of stable representations}\label{SubS:StableRep}

We consider an operator-valued function $R(z)$ with values in $\cB(\cU)$, for a complex Hilbert space $\cU$, which is analytic in a neighborhood of the unit circle $\BT$. Then $R(z)$ has a representation of the form
\begin{equation}\label{FKRreal}
R(z)=\de+z\ga_+(I_{\cX_+}-z\al_+)^{-1}\be_+ + \gamma_-(zI_{\cX_-}-\alpha_-)^ {-1}\beta_-,
\end{equation}
where $\cX_+$ and $\cX_-$ are also Hilbert spaces, $\de$, $\ga_+$, $\al_+$, $\be_+$, $\gamma_-$, $\alpha_-$, $\beta_-$ are operators acting between appropriate Hilbert spaces, and $\al_+$ and $\alpha_-$ are both exponentially stable, i.e., $\sigma(\al_+)$ and $\sigma(\alpha_-)$ are both contained in the open unit disk $\BD$: That $R(z)$ admits a representation of the form \eqref{FKRreal} follows since $R(z)$ is a $\cB(\cU)$-valued Wiener function on $\BT$, which is analytic on a neighborhood of $\BT$, and hence in $L^2(\T;\cB(\cU))$, so that for some sequence $u_k\in \ell^2(\cU)$,
\begin{equation}\label{eq:Rfourier}
	R(z)=\sum_{j=-\infty}^\infty z^jR_j.
\end{equation}
This function can be written as the sum of a $\cB(\cU)$-valued function that is analytic on a neighborhood of $\overline{\BD}$ (take terms with $j\geq0$) and an $\cB(\cU)$-valued function that is analytic on a neighborhood of $\BC\backslash \BD$ and zero at $\infty$ (for $j<0$). Then \eqref{FKRreal} follows by applying the theory of Section 4.2 in \cite{BGKR08} to both parts. A representation of the form \eqref{FKRreal} with all operators bounded and $\alpha_\pm$ exponentially stable is called a {\em stable representation} of $R(z)$. Such representations were already considered in \cite{BCR88,BR07}, and here we present the results on Wiener-Hopf factorization for rational matrix functions from \cite{FKR10} in an infinite dimensional setting. The main result of \cite{FKR10} is the following (there the result is for the less delicate setting with finite dimensions):

\begin{theorem}\label{thm:FKR10}
For a stable representation $R(z)$, which is invertible and analytic in an open neighborhood of the unit circle $\T$, the following claims are true:

The function $R(z)$ admits a right canonical Wiener-Hopf factorization if and only if there exists a $Q\in\cB(\cX_+,\cX_-)$, such that $\de-\gamma_- Q\be_+$ is invertible, which satisfies the algebraic Riccati equation
\begin{equation}\label{RiccFKR}
Q=\alpha_- Q\al_++(\beta_--\alpha_- Q\be_+)(\de-\gamma_- Q\be_+)^{-1}(\ga_+-\gamma_-Q\al_+),	
\end{equation}
and is \emph{stabilizing} in the sense that it makes the operators
\begin{equation}\label{FKRalcirc}
\begin{aligned}
\al_+^\circ&:=\al_+-\be_+(\de-\gamma_- Q\be_+)^{-1}(\ga_+-\gamma_- Q\al_+)\\
\alpha_-^\circ&:=\alpha_--(\beta_--\alpha_- Q\be_+)(\de-\gamma_- Q\be_+)^{-1}\gamma_-
\end{aligned}
\end{equation}
exponentially stable. In fact, there exists at most one stabilizing solution  $Q$ to \eqref{RiccFKR}. 

Once a stabilizing solution $Q$ to \eqref{RiccFKR} is obtained, a right canonical Wiener-Hopf factorization $R(z)=V_-(z)V_+(z)$ is given by
\begin{equation}\label{PsiThetaFKR}
V_-(z):=\de_- +\gamma_-(z I_{\cX_-}-\alpha_-)^{-1}\beta_-^\circ,\quad
V_+(z):=\de_+  + z\ga_+^\circ(I_{\cX_+}-z\al_+)^{-1}\be_+,
\end{equation}
where $\de_+,\delta_-\in\cB(\cU)$ are invertible operators such that $\delta_- \de_+=\de-\gamma_- Q\be_+$ and 
\begin{equation}\label{FKRgabecirc}
\ga_+^\circ:=\delta_-^{-1}(\ga_+-\gamma_- Q \al_+),\quad \beta_-^\circ:=(\beta_--\alpha_- Q\be_+)\de_+^{-1}.
\end{equation}
In that case, the inverses of $V_+(z)$ and $V_-(z)$ are given explicitly by
\begin{equation}\label{ThetaPsiInv}
\begin{aligned}
V_+(z)^{-1}&=\de_+^{-1}-z \de_+^{-1}\ga_+^\circ(I_{\cX_+}-z\al_+^\circ)^{-1}\be_+\de_+^{-1},\\
V_-(z)^{-1}&=\delta_-^{-1}-\delta_-^{-1}\gamma_-(z I_{\cX_-}-\alpha_-^\circ)^{-1}\beta_-^\circ \delta_-^{-1}.
\end{aligned}
\end{equation}
\end{theorem}

\begin{proof}[\bf Proof]
Most of the arguments in the proof of Theorem 1.1 from \cite{FKR10}, which establishes the above in the rational matrix function case, are algebraic in nature and it is easy to see they carry over to the infinite dimensional case presented here. Rather than going through all the arguments again, we her focus on the few claims where additional work is needed. Where nothing else is mentioned in the remainder of this proof, do as in the proof of \cite[Thm 1.1]{FKR10}.

The proof of the above claims follows by proving equivalence of the following three statements:\footnote{The equivalence of (ii) and (iii) is well known; in the infinite dimensional case considered here, the result follows from Theorem 8.6.7 in \cite{GL09}; the ``local factorizations relative to $\BT$'' assumption (see Definition 7.1.3 in \cite{GL09}) is satisfied since we assume $R(z)$ to be analytic on a neighborhood of the unit circle $\BT$.}
\begin{itemize}
  \item[(i)] There exists a stabilizing solution $Q$ to the Riccati equation \eqref{RiccFKR}.
      
  \item[(ii)] The function $R(z)$ admits a right canonical Wiener-Hopf factorization.
  
  \item[(iii)] The bounded Toeplitz operator 
 $$
 	T_R:=\bbm{ R_0&R_{-1}&R_{-2}&\cdots \\
 	R_1&R_0&R_{-1}&\cdots \\
 	R_2&R_1&R_0&\cdots \\
 	\vdots&\vdots&\vdots&\ddots },
$$
cf. \eqref{eq:Rfourier}, on $\ell^2(\zplus;\cU)$ with symbol $R$ has a bounded inverse. Then $T_R^{-1}$ will be used to construct the unique stabilizing solution $Q$ of \eqref{RiccFKR}.
\end{itemize}

The implication (ii) $\Rightarrow$ (iii) follows here because basic theory of block Toeplitz operators and the right canonical Wiener-Hopf factorization yield $T_R=T_{V_-} T_{V_+}$ and 
$$
	T_R^{-1}=T_{V_+}^{-1} T_{V_-}^{-1} =T_{V_+^{-1}} T_{V_-^{-1}},
$$ 
with $T_{V_-}$, $T_{V_+}$, $T_{V_+^{-1}}$, $T_{V_-^{-1}}$ the Toeplitz operators defined by the operator functions $V_-(z)$, $V_+(z)$, $V_+(z)^{-1}$, $V_-(z)^{-1}$, respectively; these are purely algebraic statements, which verify because of the triangularity of the Toeplitz operators.

Assume that (i) holds; we will establish (ii). The formulas \eqref{ThetaPsiInv} for $V_-(z)^{-1}$ and $V_+(z)^{-1}$ can be verified by multiplying from the left and the right with $V_-(z)$ and $V_+(z)$ in \eqref{PsiThetaFKR}. (The formulas for the inverses also follow from the standard inversion formulas in \cite[Section 2.4]{BGKR10} .) Since $\alpha^\circ_+, \alpha^\circ_-$ are both exponentially stable by assumption, it is clear that the functions $V_-(z)$ and $V_+(z)$ and their inverses are of the right type. It remains to see that $R(z)=V_-(z)V_+(z)$ on $\BT$. This follows from the computation in Part 3 of the proof of \cite[Thm 1.1]{FKR10}.

The proof that (iii) implies (i) in the infinite-dimensional case requires some modification of the proof in \cite{FKR10}: Assume that (iii) holds. As in \cite{FKR10}, since $\al_+$ and $\al_-$ are exponentially stable, we obtain well defined bounded operators by setting
\[
\cC:=\mat{\be_- & \al_- \be_- & \al_-^2\be_- \cdots}:\ell^2(\cU)\to\cX_-,\quad 
\cO:=\mat{\ga_+\\ \ga_+ \al_+\\ \ga_+ \al_+^2\\ \vdots}:\cX_+ \to \ell^2(\cU). 
\] 
Then $Q:=\cC T_R^{-1} \cO$ is a bounded operator mapping $\cX_+$ into $\cX_-$, such that $Q-\gamma_-Q\beta_+$ is invertible and $Q$ satisfies \eqref{RiccFKR} by the computations in Part 1 of the proof of \cite[Theorem 1.1]{FKR10} and the invertibility of $T$ in \cite[eq.\ (2.1)]{FKR10}. We next show that $Q$ is a stabilizing solution, i.e., that $\al_+^\circ$ and $\al_-^\circ$ are exponentially stable.

First we show that $\sigma(\al_+^\circ)\subset \overline{\mathbb{D}}$. Suppose, to the contrary, that $\lambda$ is in the boundary $\partial \sigma(\al_+^\circ)$ of the spectrum $\sigma(\al_+^\circ)$ and $|\lambda| >1$. Then $\lambda$ is in the approximate point spectrum of $\al_+^\circ$, by \cite[Prop.\ VII.6.7]{Conway90}, and so there is a sequence $\{x_n\}_{n=1}^\infty$ such that $\|x_n\|=1$ and $(\al_+^\circ -\lambda I)x_n\to 0$. By the last display on p.\ 220 of \cite{FKR10}, we have $S^*T_R^{-1}\cO =T_R^{-1}\cO\al_+^\circ$, where $S^*$ is the incoming left shift on $\ell^2(\zplus;\cU)$, and then 
$$
	(S^*-\lambda I)T_R^{-1}\cO x_n=T_R^{-1}\cO(\al_+^\circ -\lambda I)x_n \to 0.
$$
As $|\lambda| >1$ it follows that $S^*-\lambda I$ is invertible, and so $\cO x_n\to 0$. In particular, $\ga_+ x_n\to 0$, and $\cO\al_+x_n\to 0$. From the latter it follows that 
$$
	Q\al_+x_n=\cC T_R^{-1}\cO\al_+x_n\to 0.
$$
From these we obtain
\[
(\al_+ -\lambda I)x_n=(\al_+^\circ -\lambda I)x_n +\be_+(\de -\ga_-Q\be_+)^{-1}(\ga_+ -\ga_- Q\al_+)x_n \to 0.
\]
Thus, $\lambda$ is in the approximate point spectrum of $\al_+$. But $\al_+$ is exponentially stable, and this means that $|\lambda | <1$. This contradicts the assumption that $|\lambda| >1$. Hence we see that the spectrum of $\al_+^\circ$ is in the closed unit disc. 

The next step is to show that $\sigma(\alpha_+^\circ)\cap\T=\emptyset$, and we do this via establishing the invertibility of $V_+$ on an open disc which contains $\overline\D$. Since $\si(\al_+^\circ)\subset \ov{\BD}$, it follows that the first formula in \eqref{ThetaPsiInv} is well-defined for $z$ inside the open unit disc and gives the inverse of $V_+(z)$ there. Now note that by the standing assumptions on $R$ and the exponential stability of $\alpha_\pm$, we can find an annulus $r_1< |z| < r_2$ with $r_1<1$ and $r_2>1$, where $R(z)$, $R(z)^{-1}$, $V_+(z)$ and $V_-(z)$ are all analytic. The factorization $R(z)=V_-(z)V_+(z)$ is verified in Part 3 of the proof of \cite[Thm 1.1]{FKR10}, and we then have for $r_1<|z|<1$ that $V_+(z)^{-1}= R(z)^{-1}V_-(z)$, so not only $I=(R(z)^{-1}V_-(z))V_+(z)$ but also $V_+(z)(R(z)^{-1}V_-(z))=I$. However, both of these equalities extend analytically to the annulus $r_1<|z|<r_2$. Thus we have that $V_+(z)$ is invertible in the open disc $|z|<r_2$, which clearly contains $\overline\D$. 

Now we are ready to prove that $\al_+^\circ$ is exponentially stable, by showing that $\si(\al_+^\circ)\subset \BD$. Consider the $2 \times 2$ operator block polynomial 
\[
M(z):=\mat{I_{\cX_+}-z\al_+ & -\be_+\\ z\ga_+^\circ & \de_+}. 
\]
By construction, $\de_+$ is invertible, and since $\al_+$ is exponentially stable, $I_{\cX_+}-z\al_+$ is invertible for all $z$ in an open neighborhood of $\overline{\BD}$. The Schur complement of $M(z)$ with respect to $\de_+$ is 
\[
M(z)\backslash \de_+= I_{\cX_+}-z\al_+ + z\be_+ \de_+^{-1}\ga_+^\circ = I_{\cX_+}-z (\al_+ -\be_+\de_+^{-1} \de_-^{-1} \de_- \ga_+^\circ)=I_{\cX_+}-z \al_+^\circ.
\]
For $z\in \overline{\BD}$, the Schur complement of $M(z)$ with respect to $I_{\cX_+}-z\al_+$ is    
\[
M(z)\backslash(I_{\cX_+}-z\al_+)= \de_+ + z \ga_+^\circ(I_{\cX_+}-z\al_+)^{-1}\be_+ =V_+(z). 
\]
Hence for $z\in \overline{\BD}$, $V_+(z)$ and $I_{\cX_+}-z \al_+^\circ$ are Schur coupled (cf., \cite{BT94}), and as a consequence  $I_{\cX_+}-z \al_+^\circ$ is invertible if and only if $V_+(z)$ is invertible. However, as established above, $V_+(z)$ has invertible values even on a closed neighborhood of $\overline{\BD}$, which proves that $\overline\D\subset\rho(\alpha_+^\circ)^{-1}$, and hence that $\al_+^\circ$ is exponentially stable. (For this argument, compare also \cite{BGKR10}, Theorem 5.7 and its proof.

In order to obtain that $\alpha_-^\circ$ is also exponentially stable, define $R^\sharp(z):=R(1/\overline z)^*$, $1/\overline z\in\dom R$. The reader may next verify that from a stable representation \eqref{FKRreal} of $R$, one gets a stable representation
$$
	R^\sharp(z)=\de^*+\be_+^*(zI_{\cX_+}-\al_+^*)^{-1}\ga_+^* + z\be_-^*(I_{\cX_-}-z\alpha_-^*)^ {-1}\ga_-^*,
$$
and $T_{R^\sharp}=T_R^*$ is then bounded on $\ell^2(\zplus;\cU)$ with bounded inverse. The above argument then shows that $(\alpha_+^\circ)^\sharp$ is exponentially stable, and now it only remains to note that \eqref{FKRalcirc} yields
$$
	(\alpha_+^\circ)^\sharp = \al_-^*-\ga_-^*(\delta^*-\be_+^*Q^\sharp\ga_-^*)^{-1}(\be_-^*-\be_+^*Q^\sharp\al_-^*)=(\al_-^\circ)^*,
$$
since $Q^\sharp=Q^*$.

Finally, uniqueness of the stabilizing solution $Q$ of \eqref{RiccFKR} follows by the same argument as in Part 4 of the proof of Theorem 1.1 in \cite{FKR10}.
\end{proof}

From the procedure to determine a right canonical Wiener-Hopf factorization for $R(z)$ it is easy to derive a procedure for a canonical left Wiener-Hopf factorization for $R(z)$, namely by applying the right canonical Wiener-Hopf factorization procedure to $\wtilR(z):=R(\frac{1}{z})$, which is again analytic with an analytic inverse in some annulus containing $\T$. The result of this procedure is as follows:

\begin{corollary}
The function $R(z)$ admits a left canonical Wiener-Hopf factorization if and only if the algebraic Riccati equation\begin{equation}\label{RiCCFKR-L}
\wtilQ=\al_+ \wtilQ\alpha_- +(\be_+- \al_+\wtilQ \beta_-)(\de-\ga_+ \wtilQ \beta_-)^{-1}(\gamma_- - \ga_+ \wtilQ \alpha_-)	
\end{equation}
has a solution $\wtilQ\in\cB(\cX_-,\cX_+)$ which is {\em stabilizing} in the sense that the operators
\begin{equation}\label{FKRalcirc-L}
\begin{aligned}
\wtil{\al}_-^\circ&:=\al_- - \be_- (\de-\ga_+ \wtilQ \beta_-)^{-1}(\gamma_- - \ga_+ \wtilQ \alpha_-)\\
\wtil{\alpha}_+^\circ&:=\al_+-(\be_+- \al_+\wtilQ \beta_-)(\de-\ga_+ \wtilQ \beta_-)^{-1}\ga_+
\end{aligned}
\end{equation}
are exponentially stable. There exists at most one stabilizing solution $\wtilQ$ to \eqref{RiCCFKR-L}. 

With a stabilizing solution $\wtilQ$ to \eqref{RiCCFKR-L}, a left canonical Wiener-Hopf factorization $R(z)=W_+(z)W_-(z)$ for $R(z)$ is obtained by taking
\begin{equation}\label{XiPhiFKR-L}
\begin{aligned}
W_+(z)&:=\wtil{\de}_+ + z\ga_+ (I_{\cX_+}-z\al_+)^{-1} \wtil{\beta}_+^\circ,\\
W_-(z)&:=\wtil{\delta}_-+\wtil{\ga}_-^\circ (zI_{\cX_-}-\al_-)^{-1}\be_-,
\end{aligned}
\end{equation}
where $\wtil{\de}_+,\wtil{\delta}_-\in\cB(\cU)$ are invertible operators such that $\wtil{\de}_+\wtil{\delta}_-=\de-\ga_+ \wtilQ \beta_-$ and
\begin{equation}\label{FKRgabecirc-L}
\wtil{\ga}_-^\circ:=\wtil{\de}_+^{-1}(\gamma_--\ga_+ \wtilQ \alpha_-),\quad \wtil{\beta}_+^\circ:=(\be_+-\al_+ \wtilQ \beta_-)\wtil{\delta}_-^{-1}
\end{equation}
The inverses of $W_+(z)$ and $W_-(z)$ are then given explicitly by
\begin{equation}\label{XiPhiInv-L}
\begin{aligned}
W_+(z)^{-1}&=\wtil{\de}_+^{-1}-z\wtil{\de}_+^{-1}\ga_+ (I_{\cX_+}-z\wtil{\al}_+^\circ)^{-1} \wtil{\beta}_+^\circ \wtil{\de}_+^{-1},\\
W_-(z)^{-1}&=\wtil{\delta}_-^{-1}- \wtil{\delta}_-^{-1}\wtil{\ga}_-^\circ(zI_{\cX_-}-\wtil{\alpha}_+^\circ)^{-1}\beta_-\wtil{\delta}_-^{-1}.
\end{aligned}
\end{equation}  
\end{corollary}

\begin{proof}[\bf Proof]
First observe that one gets a stable representation of $\widetilde R$ by swapping all plus and minus signs in a stable representation \eqref{FKRreal} of $R$:
\begin{align*}
\wtilR(z) 
&=\de+\mbox{$\frac{1}{z}$}\ga_+(I_{\cX_+}-\mbox{$\frac{1}{z}$}\al_+)^{-1}\be_+ + \gamma_-(\mbox{$\frac{1}{z}$}I_{\cX_-}-\alpha_-)^ {-1}\beta_-\\
&=\de+ z\gamma_-(I_{\cX_-}-z\alpha_-)^ {-1}\beta_-+\ga_+(zI_{\cX_+}-\al_+)^{-1}\be_+.
\end{align*}
Formulas \eqref{RiCCFKR-L}, \eqref{FKRalcirc-L} and \eqref{FKRgabecirc-L} are obtained using such a sign swap in \eqref{RiccFKR}, \eqref{FKRalcirc} and \eqref{FKRgabecirc}, respectively.

Now, applying Thm \ref{thm:FKR10} to $\widetilde R$, we get the result, observing that $\wtilR(z)=\wtil{V_-}(z)\wtil{V_+}(z)$ is a right canonical Wiener-Hopf factorization for $\wtilR(z)$ if and only if $R(z)=\wtil{V_-}(\frac{1}{z})\wtil{V_+}(\frac{1}{z})$ is a left canonical Wiener-Hopf factorization for $R(z)$. Now simply pick $W_\pm(z):=\wtil{V_\mp}(\mbox{$\frac{1}{z}$})$ to get \eqref{XiPhiFKR-L} and \eqref{XiPhiInv-L} from \eqref{PsiThetaFKR} and \eqref{ThetaPsiInv}.
 \end{proof}

\subsection{Canonical Wiener-Hopf factorization for dichotomous realizations}\label{SubS:DichotReal}

In the second approach to canonical Wiener-Hopf factorization, we assume that the $\cB(\cU)$ valued function $R(z)$ is given by a transfer function formula of the form 
\begin{equation}\label{trans}
R(z)=D+z C(I_\cX-zA)^{-1}B,
\end{equation}
with $\cX$ another Hilbert space and $A,B,C$ and $D$ operators acting between appropriate Hilbert spaces, and here we assume that $\si(A)\cap \BT =\emptyset$. Such functions are analytic in a neighborhood of 0, in addition to a neighborhood of $\BT$. The right-hand side in \eqref{trans} is the transfer function of the discrete-time input-state-output system $\Sigma$ over the nonnegative integers $\BZ_+$ given by 
\begin{equation}\label{DichotSys}
\Si: 
	\left\{
		\begin{aligned}
			x(n+1) &= A x(n)+ Bu(n) \\	
			y(n) &= Cx(n)+Du(n),\quad n\in\zplus,
			\quad x(0)=x_0~\text{given},
		\end{aligned}
	\right.
\end{equation} 
which is called {\em dichotomous} since we assume $\si(A)\cap \BT =\emptyset$; see \cite{BGtH18c} for more on linear discrete-time input-state-output dichotomous systems, and see \cite{tHKR25} for Wiener-Hopf factorization in case $G(z)$ has the form $G(z)=I+K(z)$ with $K(z)$ strictly contractive on $\BT$. The dichotomy assumption is equivalent to the existence of a direct sum decomposition  $\cX=\cX_- \dotplus \cX_+$ of the state space $\cX$, such that with respect to this decomposition, $A$ has the form
\begin{equation}\label{DichotDecompA}
	A=\bbm{A_-&0\\0&A_+}: \cX_-\dotplus\cX_+\to\cX_-\dotplus\cX_+,
\end{equation}
with $A_-$ invertible, and $A_+$ and $A_-^{-1}$ both exponentially stable. The spaces $\cX_+$ and $\cX_-$ are in fact uniquely determined by $A$, and correspond to the spectral subspaces of $A$ with respect to $\BD$ and $\BC \backslash \ov{\BD}$, respectively. Next we decompose $C$ and $B$ with respect to the decomposition $\cX=\cX_-\dotplus\cX_+$ as
\begin{equation}\label{DichotDecompBC}
C=\mat{C_-&C_+}:\cX_-\dotplus\cX_+\to \cU \quad\mbox{and}\quad B=\mat{B_-\\B_+}:\cU\to \cX_-\dotplus\cX_+.
\end{equation}
In order to apply the matching invariant subspace approach, we shall assume that the feedthrough operator $D$ is invertible, and define 
\begin{equation}\label{Across}
A^\times := A-BD^{-1}C.
\end{equation}
In that case $R(z)$ is invertible on the set $\{0\}\cup (\rho(A)^{-1}\cap \rho(A^\times)^{-1})$, with inverse given by 
\[
R(z)^{-1}=D^{-1}-z D^{-1}C (I-z A^\times)^{-1}BD^{-1},\quad z\in \{0\}\cup (\rho(A)^{-1}\cap \rho(A^\times)^{-1}).
\]
When $A^\times$ is also dichotomous, then for $A^\times$ we also have a decomposition $\cX=\cX_+^\times \dotplus \cX_-^\times$, with $\cX_+^\times$ and $\cX_-^\times$ the spectral subspaces of $A^\times$ with respect to $\BD$ and $\BC \backslash \ov{\BD}$, respectively, such that $A^\times$ decomposes as 
\begin{equation}\label{DichotDecompAcross}
A^\times=\mat{A^\times_-&0\\0& A^\times_+}: \cX_-^\times\dotplus\cX_+^\times\to\cX_-^\times\dotplus\cX_+^\times.
\end{equation}

Now assume that $D$ is invertible and factor $D=D_+ D_-$ with $D_\pm\in\cB(\cU)$ invertible. Set 
$$
	\whatR(z):=D_+^{-1}R\left(\frac1z\right)D_-^{-1}=I+D_+^{-1}C(zI_\cX-A)^{-1}BD_-^{-1},
$$
so that a canonical Wiener-Hopf factorization of $R$ can easily be obtained from such a factorization of $\whatR$ using the formula $R(z)=D_+ \whatR(\frac{1}{z}) D_-$. Note, however, that a right canonical factorization of the operator-valued function $\whatR$ translates to a left canonical factorization of $R(z)$ (and vice versa), due to the change of $\frac{1}{z}$ to $z$. Applying Theorem 7.1 in \cite{BGKR10} to $\whatR$, with the Cauchy contour $\Ga:=\BT$, we thus get the following result: 

\begin{corollary}
The function $R$ in \eqref{trans} admits a left canonical factorization if and only if  both $A$ and $A^\times$ are dichotomous and with respect to the state space decompositions $\cX_- \dotplus \cX_+=\cX=\cX_-^\times \dotplus \cX_+^\times$ above, we have the \emph{matching decomposition} $\cX=\cX_-^\times\dotplus \cX_+$; then write $P$ for the projection of $\cX$ along $\cX_+$ onto $\cX_-^\times$. 

When the above conditions hold, for every factorization $D=D_+D_-$ with $D_\pm$ invertible, $R(z)=W_+(z)W_-(z)$ is a left canonical Wiener-Hopf factorization, where
\begin{equation}\label{XiPhiDichot}
\begin{aligned}
W_+(z):=&\,D_+ + zC(I_\cX-zA)^{-1}(I_\cX-P)BD_-^{-1},\\
W_-(z):=&\,D_- + zD_+^{-1}CP(I_\cX-zA)^{-1}B,\quad z\in\T,
\end{aligned}
\end{equation}
with inverses given by 
\begin{equation}\label{XiPhiDichotInv}
\begin{aligned}
W_+(z)^{-1}&=D_+^{-1} - zD_+^{-1}C(I_\cX-P)(I_\cX-zA^\times)^{-1}BD^{-1},\\  
W_-(z)^{-1}&=D_-^{-1} - zD^{-1}C(I_\cX-zA^\times)^{-1}PBD_-^{-1}, \quad z\in\T.
\end{aligned}
\end{equation}
\end{corollary}

Note that the invariances of $\cX_+$ under $A$ and $\cX_-^\times$ under $A^\times$ imply that, with respect to the decomposition $\cX=\cX_-^\times\dotplus \cX_+$, we have the decompositions $P=\sbm{I&0\\0&0}$,
\[
A=\mat{A_{11}&0 \\ A_{21} & A_+},\quad A^\times =\mat{A_-^\times& A_{12}^\times \\ 0&A_{22}^\times},\quad B=\mat{B_-^\times\\\wtilB_+},\quad C=\mat{C_-^\times&\wtilC_+} 
\] 
with $A_+$ as in \eqref{DichotDecompA} and $A_-^\times$ as in \eqref{DichotDecompAcross}. Note that $\wtilC_+=C_+$. It then follows from \cite[Lemma 5.9]{BGKR08} that $\si(A_{11})\subset \BC\backslash \ov{\BD}$ and $\si(A_{22}^\times)\subset \BD$, and that $W_\pm(z)$ and their inverses are also given by 
\begin{equation}\label{eq:WpmInvariant}
\begin{aligned}
W_+(z)&=D_+ + zC_+(I_{\cX_+}-zA_+)^{-1}\wtilB_+D_-^{-1},\quad z\in\overline \D,\\ 
W_+(z)^{-1}&=D_+^{-1} - zD_+^{-1}C_+(I_{\cX_+}-zA_{22}^\times)^{-1}\wtilB_+D^{-1},\quad z\in\overline \D,\\ 
W_-(z)&=D_- + zD_+^{-1}C_-^\times(I_{\cX_-^\times}-zA_{11})^{-1}B_-^\times,\quad z\in\C\setminus\D,\\ 
W_-(z)^{-1}&=D_-^{-1} - zD^{-1}C_-^\times (I_{\cX_-^\times}-zA_-^\times)^{-1}B_-^\times D_-^{-1},\quad z\in\C\setminus\D,
\end{aligned}
\end{equation}
confirming that $W_+(z)$ and $W_+(z)^{-1}$ (that $W_-(z)$ and $W_-(z)^{-1}$) have analytic extensions to open neighborhoods of $\overline\D$ (of $\C\setminus\D$). Hence, $R(z)=W_+(z)W_-(z)$ is indeed a left canonical Wiener-Hopf factorization.

In order to obtain a right canonical Wiener-Hopf factorization of a function $R(z)$ given by a dichotomous realization \eqref{trans} it is necessary and sufficient for $A$ and $A^\times$ to be dichotomous and for $\cX$ to decompose as $\cX=\cX_-\dotplus \cX_+^\times$. In that case, write $P^\times$ for the projection on $\cX$ along $\cX_-$ onto $\cX_+^\times$. As above, we get a right canonical Wiener-Hopf factorization $R(z)=V_-(z)V_+(z)$ of $R(z)$ with the factors and their inverses obtained by replacing $P$ by $P^\times$ in \eqref{XiPhiDichot} and \eqref{XiPhiDichotInv}, in such a way that $V_\pm$ is obtained from the formula for $W_\mp$, and swapping $D_+$ and $D_-$:
\begin{equation}\label{PsiThetaDichot}
\begin{aligned}
{V_-}(z)&:=D_- + zC(I_\cX-zA)^{-1}(I_\cX-P^\times)BD_+^{-1},\\ 
{V_+}(z)&:=D_+ + zD_-^{-1}CP^\times(I_\cX-zA)^{-1}B,
\end{aligned}
\end{equation}
with inverses given by 
\begin{equation}\label{PsiThetaDichotInv}
\begin{aligned}
{V_-}(z)^{-1}&=D_-^{-1} - zD_-^{-1}C(I_\cX-P^\times)(I_\cX-zA^\times)^{-1}BD^{-1},\\  
{V_+}(z)^{-1}&=D_+^{-1} + zD^{-1}C(I_\cX-zA^\times)^{-1}P^\times BD_+^{-1}.
\end{aligned}
\end{equation}
Again, using the invariance of $\cX_-$ and $\cX_+^\times$, it follows that with respect to the matching decomposition $\cX=\cX_-\dotplus \cX_+^\times$, $P^\times=\sbm{0&0\\0&I}$, 
\begin{equation}%\label{ABCdec+}
A =\mat{A_- & A_{12}\\ 0 & A_{22}},\quad 
A^\times=\mat{A_{11}^\times&0\\A_{21}^\times & A_+^\times},\quad 
B=\mat{\widetilde B_-\\ B_+^\times},\quad 
C=\mat{C_- & C_+^\times},
\end{equation}
with $A_+$, $A_-^\times$ and $C_-$ as in the decompositions \eqref{DichotDecompA}, \eqref{DichotDecompAcross} and \eqref{DichotDecompBC}. As before, $\si(A_{22})\subset \BC\backslash \ov{\BD}$ and $\si(A_{11}^\times)\subset \BD$, and the right canonical version of \eqref{eq:WpmInvariant} is
\begin{equation}\label{eq:RightWHinvariant}
\begin{aligned}
V_-(z)&=D_- + zC_-(I_{\cX_-}-zA_-)^{-1}\wtilB_-D_+^{-1},\quad z\in\overline \D,\\ 
V_-(z)^{-1}&=D_-^{-1} - zD_-^{-1}C_-(I_{\cX_+}-zA_{11}^\times)^{-1}\wtilB_-D^{-1},\quad z\in\overline \D,\\ 
V_+(z)&=D_+ + zD_-^{-1}C_+^\times(I_{\cX_-^\times}-zA_{22})^{-1}B_+^\times,\quad z\in\C\setminus\D,\\ 
V_+(z)^{-1}&=D_+^{-1} - zD^{-1}C_+^\times (I_{\cX_-^\times}-zA_+^\times)^{-1}B_+^\times D_+^{-1},\quad z\in\C\setminus\D.
\end{aligned}
\end{equation}
Under the assumptions made in this paragraph, the formulas in \eqref{eq:RightWHinvariant} give a right canonical Wiener-Hopf factorization $R(z)=V_-(z)V_+(z)$, for all factorizations $D=D_-D_+$ with $D_\pm$ invertible.
 
Following Section 12.3 in \cite{BGKR10} we can also associate a Riccati equation with the above paired invariant subspace approach. We include here the details for the right canonical Wiener-Hopf factorization.
We start from the decomposition $\cX=\cX_- \dotplus \cX_+$ and the corresponding decompositions of $A,$ $B$ and $C$ in \eqref{DichotDecompA}--\eqref{DichotDecompBC}. Then 
\begin{equation}\label{eq:AtimesSplit}
	A^\times =\mat{A_- - B_- D^{-1} C_- & -B_- D^{-1}C_+ \\ -B_+D^{-1}C_- & A_+ -B_+D^{-1}C_+ }.
\end{equation}
We next apply Theorem 12.5 (and its proof) in \cite{BGKR10} to $R(\frac{1}{z})$ in \eqref{trans}, adjusted to the left canonical Wiener-Hopf factorization setting. %\footnote{MK: These are the details that I was hoping for: I did first not understand what ``adjusted to the left canonical Wiener-Hopf'' factorization means, and what do the formulas look like for this adjusted setting, but now I see that this is compactly explained between the result statement and the proof in OT200. 
%
%But is Thm 12.5 in OT200 really the right reference? To me that result seems to say something else than you need for your discussion, and it seems to me like your discussion is part of the proof of Thm 12.5 in OT200. Where do we have a result saying that ``right canonical WH factorization exists if and only if we have correct matching of spectral subspaces''? We need this also in \S\ref{S:LeftRight}.}
It follows that $R(z)$ admits a right canonical Wiener-Hopf factorization if and only if $\cX={\rm Im\,} \sbm{I \\ 0} \dot+ \cX_+^\times$, see also Theorem 6.1 in \cite{BGKR08}. This is equivalent to saying that $\cX_+^\times ={\rm Im\,}\sbm{ \whatQ \\ I}$ for some bounded operator $\widehat Q$. Invariance of $\cX_+^\times$ under $A^\times$ in \eqref{eq:AtimesSplit} is equivalent to 
 $\whatQ\in\cB(\cX_-,\cX_+)$ being a solution to the Riccati equation %\footnote{Andr\'e will put in a little bit more detail.}
\begin{equation}\label{RiccDichot}
\whatQ B_+ D^{-1} C_- \whatQ - \whatQ(A_+-B_+ D^{-1} C_+) + (A_- -B_-D^{-1}C_-)\whatQ -B_-D^{-1}C_+=0.
\end{equation}
Setting $S=\mat{ I & \whatQ \\ 0 & I}$ and computing $S^{-1}A^\times S$ using the Riccati equation, one sees that $\cX_+^\times$ is the spectral subspace of $A^\times$ corresponding to the open unit disc if and only if the following two spectral inclusions hold
\begin{equation}\label{StabCond+}
\si(A_+ - B_+ D^{-1} (C_+ + C_- \whatQ))\subset \BD \mbox{ and }
\si(A_- - (B_- - \whatQ B_+)D^{-1} C_-)\subset \BC \backslash \ov{\BD}.
\end{equation}
For later purposes, we point out that the Riccati equation \eqref{RiccDichot} can be rewritten as  
\begin{equation}\label{RiccDichot2}
A_-\whatQ - \whatQ A_+ + (\whatQ B_+ -B_-)D^{-1}(C_-\whatQ + C_+)=0.
\end{equation}

\subsection{Dichotomous realizations versus stable representations}\label{SubS:Connect}

In order to make a connection between the two approaches to canonical Wiener-Hopf factorization, we need to assume that $R(z)$ is defined and analytic in a neighborhood of $0$. To achieve this in the stable representation \eqref{FKRreal}, we shall assume that $\al_-$ is invertible, in which case $R(0)=\de - \ga_- \al_-^{-1}\be_-$.

\begin{proposition}\label{P:FKRdichot}
An operator-valued function $R(z)$ is the transfer function of a dichotomous system if and only if it has a stable representation \eqref{FKRreal} with $\alpha_-$ invertible. In case $R(z)$ is given by \eqref{FKRreal} with $\al_+$ and $\alpha_-$ exponentially stable and $\alpha_-$ invertible, then $R(z)$ is the transfer function of the dichotomous system \eqref{DichotSys} with
\begin{equation}\label{FKRtoDichot}
\mat{A&B\\C&D}:=\left[ \begin{array}{cc|c} A_-&0&B_-\\0&A_+&B_+\\ \hline C_-&C_+&D \end{array} \right]
=\left[ \begin{array}{cc|c} \alpha_-^{-1}&0& \alpha_-^{-1}\beta_-\\0&\al_+&\be_+\\ \hline -\gamma_- \alpha_-^{-1}&\ga_+&\de-\gamma_- \alpha_-^{-1}\beta_- \end{array} \right].
\end{equation}
In particular, $R(z)$ does not have a pole at $0$. Conversely, if $R(z)$ is the transfer function of the dichotomous system \eqref{DichotSys} with $A$ decomposing as in \eqref{DichotDecompA} and $B$ and $C$ as in \eqref{DichotDecompBC}, then
$R(z)$ is given by \eqref{FKRreal} with
\begin{equation}\label{DichotToFKR}
\begin{aligned}
\de:=D-C_-A_-^{-1}B_-,\quad \ga_+:=& C_+,\quad  \al_+:=A_+,\quad  \be_+:=B_+,\\ 
\gamma_-:=-C_-A_-^{-1},\quad \alpha_-:=& A_-^{-1},\quad \beta_-:=A_-^{-1}B_-.
\end{aligned}
\end{equation}
\end{proposition}

\begin{proof}[\bf Proof]
The conversions in \eqref{FKRtoDichot} and \eqref{DichotToFKR} between the operators $A,B,C,D$ and $\al_+,\be_+,\ga_+,\de$,$\alpha_-,\beta_-,\gamma_-$, with $\alpha_-$ and $A_-$ invertible, are each others inverses, as is easily checked.  Assume that $A$, $B$, $C$, $D$, $\al_+$, $\be_+$, $\ga_+$, $\de$, $\alpha_-$, $\beta_-$ and $\gamma_-$, with $\alpha_-$ and $A_-$ invertible, are related as in \eqref{FKRtoDichot} and \eqref{DichotToFKR}; then
\begin{align*}
D+zC(I-zA)^{-1}B &=D + zC_+(I-zA_+)^{-1}B_+ + zC_-(I-zA_-)^{-1}B_-\\
&=D + z\ga_+(I-z\al_+)^{-1}\be_+ -z\gamma_- \alpha_-^{-1}(I-z\alpha_-^{-1})^{-1}\alpha_-^{-1}\beta_-\\
&=D+\gamma_-\alpha_-^{-1}\beta_-  + z\ga_+(I-z\al_+)^{-1}\be_+\\
&\qquad - \gamma_-(I-z\alpha_-^{-1})^{-1}\alpha_-^{-1}\beta_-\\
&=\de + z\ga_+(I-z\al_+)^{-1}\be_+ - \gamma_-(\alpha_--zI)^{-1}\beta_-\\
&=\de + z\ga_+(I-z\al_+)^{-1}\be_+ + \gamma_-(zI-\alpha_-)^{-1}\beta_-.
\end{align*}
This calculation shows that the transfer function \eqref{trans} of the dichotomous system \eqref{DichotSys} and the stable representation \eqref{FKRreal} are the same function. Finally, $A_+=\alpha_+$ and $A_-^{-1}=\alpha_-$ obviously have the required matching stability properties.
\end{proof}

%%%%%%%%%%%%%%%%%%%%%%%%%%%%%%%%%%%%%%%%%%%%%%%%%%%%%%%%%%%%%%%%%%%%%%%%%%%%%%%%%%%%%%%%%%
\section{Analysis of the associated Riccati equations}\label{S:Ricc}

In this section we compare the Riccati equations that appear in the two approaches that were discussed in the previous section, for right canonical Wiener-Hopf factorization. These are \eqref{RiccFKR} in the setting of stable representations and \eqref{RiccDichot2} in the setting of dichotomous realizations. In order to compare the two Riccati equations, we work in the setting of a stable representation \eqref{FKRreal} with the additional assumption that $\al_-$ is an invertible operator, so that we can also work in the setting of dichotomous realizations via the connection provided by Proposition \ref{P:FKRdichot}. In that case, the Riccati equation \eqref{RiccDichot2}, in terms of the stable representation operators, after multiplying with $\alpha_-$ on the left, translates to 
\begin{equation}\label{RiccFKR2}
Q=\al_- Q\al_+ + (\be_--\al_- Q\be_+) (\de-\ga_-\al_-^{-1}\be_-)^{-1}(\ga_+ - \ga_-\al_-^{-1} Q).
\end{equation}
Conforming to \eqref{StabCond+} and \eqref{FKRtoDichot}, we call a solution $Q$ of \eqref{RiccFKR2} \emph{stabilizing} if the operators
\begin{equation}\label{RiccFKR2stableOps}
\begin{aligned}
	\alpha_+-\beta_+(\delta-\gamma_-\alpha_-^{-1}\beta_-)^{-1}(\gamma_+-\gamma_-\alpha_-^{-1}Q)
	\quad\text{and}&\\
	\alpha_-^{-1} + \alpha_-^{-1}(\beta_--\alpha_-Q\beta_+)(\delta-\gamma_-\alpha_-^{-1})^{-1}\gamma_-\alpha_-^{-1}&
\end{aligned}
\end{equation}
are stable, and anti-stable, respectively.

It turns out that in general the solution sets of \eqref{RiccFKR} and \eqref{RiccFKR2} need not coincide unless we restrict to stabilizing solutions. Indeed, for $Q\in\cB(\cX_+,\cX_-)$ to satisfy \eqref{RiccFKR}, the operator $\de-\ga_- Q \be_+$ needs to be invertible, which is not a requirement in \eqref{RiccFKR2}. On the other hand, for \eqref{RiccFKR2} necessarily $\de-\ga_-\al_-^{-1}\be_-=R(0)$ should be invertible, while one might be able to solve \eqref{RiccFKR} without this condition. We illustrate this in the following example. 

\begin{example}\label{E:RiccEqCounter}
We show that solution sets of the two Riccati equations need not coincide; in fact, neither solution set needs to be contained in the other.

The first example will show that it is possible that \eqref{RiccFKR} has a stabilizing solution while \eqref{RiccFKR2} does not, because $R(0)=\de-\ga_-\al_-^{-1}\be_-$ is not invertible. Choose $\alpha_+:=0$, $\beta_+:=\gamma_+:=1$, let $\alpha_-:=a\in(0,1)$, $\beta_-:=\gamma_-:=1$, and choose $\delta:=\tfrac{1}{a}$. Then \eqref{RiccFKR} becomes the equation $q=(1-aq)(\tfrac{1}{a}-q)^{-1}$. One easily checks that $q=a$ is a solution, for which we have $\alpha_-^\circ =0$ (compare also Corollary \ref{Cor:alpha-circ} below) and $\alpha_+^\circ=(a-\tfrac{1}{a})^{-1}$. Note that for $a$ small enough $\alpha_+^\circ$ will be stable, but $R(0)=0$ is not invertible.

The second example will show that it is possible that \eqref{RiccFKR2} has two solutions while \eqref{RiccFKR} has only one solution. Take $\beta_+:=1, \gamma_-:=1$, $\alpha_+$ and $\alpha_-$ scalar and non-zero, choose $\delta:=\tfrac{\gamma_+}{\alpha_+}$, and take $\beta_-, \gamma_+$ such that $\gamma_+\alpha_- -\beta_-\alpha_+ \not=0$. Then \eqref{RiccFKR} becomes
\[
q=\alpha_-\alpha_+ q +(\beta_- -\alpha_- q)(\tfrac{\gamma_+}{\alpha_+} -q)^{-1}(\gamma_+ -q\alpha_+).
\]
This is equivalent to $q=\alpha_- \alpha_+ q +\alpha_+\beta_- -\alpha_- \alpha_+ q=\alpha_+\beta_-$. So \eqref{RiccFKR} only has one solution. On the other hand, \eqref{RiccFKR2} becomes
\[
q=\alpha_-\alpha_+ q +(\beta_- -\alpha_- q)(\tfrac{\gamma_+}{\alpha_+} -\tfrac{\beta_-}{\alpha_-})^{-1}(\gamma_+-\tfrac{1}{\alpha_-}q).
\]
Multiply this equality by $\tfrac{\gamma_+\alpha_--\beta_-\alpha_+}{\alpha_-\alpha_+}$ in order to obtain
\[
\tfrac{\gamma_+\alpha_--\beta_-\alpha_+}{\alpha_-\alpha_+}(1-\alpha_-\alpha_+)q
= (\beta_- -\alpha_- q)(\gamma_+-\tfrac{1}{\alpha_-}q),
\]
or equivalently,
\[
\left(
\tfrac{\gamma_+}{\alpha_+} -\tfrac{\beta_-}{\alpha_-} -\gamma_+\alpha_-+\beta_-\alpha_+\right) q=
q^2-\alpha_-\gamma_+ q -\tfrac{\beta_-}{\alpha_-} q+\beta_-\gamma_+,
\]
which in turn equivalent is to 
\[
0=q^2-\left(\tfrac{\gamma_+}{\alpha_+}+\beta_-\alpha_+\right)q +\beta_-\gamma_+ 
=\left(q-\tfrac{\gamma_+}{\alpha_+}\right)\left(q-\beta_-\alpha_+\right).
\]
This shows that there are two solutions to \eqref{RiccFKR2}, of which one coincides with the unique solution to \eqref{RiccFKR}, in line with Proposition \ref{P:SolSetsRels} below.\hfill $\Box$
\end{example}

In the next two lemmas we investigate the two invertibility requirements, starting with the invertibility of $R(0)$. 

\begin{lemma}\label{L:R(0)inv}
Let $R(z)$ be given by the stable representation \eqref{FKRreal} with $\alpha_-$ invertible. Then the following two implications hold:
\begin{enumerate}
\item If $R(0)=\de-\gamma_-\alpha_-^{-1}\beta_-$ is invertible, then $\alpha_-^\circ$ is invertible for all $Q\in\cB(\cX_+,\cX_-)$ such that $\delta-\gamma_-Q\beta_+$ is invertible, and we have the inversion formula
\begin{equation}\label{alphacirc-inv}
(\alpha_-^{\circ})^{-1}=\alpha_-^{-1}+ \alpha_-^{-1}(\beta_--\alpha_- Q\be_+)(\de-\gamma_-\alpha_-^{-1}\beta_-)^{-1}\gamma_-\alpha_-^{-1}.
\end{equation}
\item If there exists a $Q\in\cB(\cX_+,\cX_-)$ such that $\delta-\gamma_-Q\beta_+$ and $\alpha_-^\circ$ are invertible, then $R(0)$ is invertible.
\end{enumerate}
\end{lemma}

\begin{proof}[\bf Proof]
For $Q\in\cB(\cX_+,\cX_-)$ define 
\[
M_Q:=\mat{\alpha_-& \beta_--\alpha_- Q \be_+\\\gamma_- & \de-\gamma_- Q \be_+}.
\]
Then the Schur complement of $M_Q$ with respect to $\al_-$ is equal to 
\begin{align*}
M_Q/\alpha_- & = \de-\gamma_- Q \be_+ - \gamma_- \alpha_-^{-1} (\beta_--\alpha_- Q \be_+)= \de- \gamma_- \alpha_-^{-1}\beta_- =R(0).
\end{align*}
In particular, $M_Q$ is invertible if and only if $R(0)$ is invertible. Hence, invertibility of $M_Q$ is independent of the choice of $Q$. 

Now assume that $Q\in\cB(\cX_+,\cX_-)$ is such that $\de-\ga_- Q \be_+$ is invertible. Then the Schur complement of $M_Q$ with respect to $\de-\ga_- Q \be_+$ works out as 
\begin{align*}
M_Q/(\de-\ga_- Q \be_+) &= \alpha_- - (\beta_--\alpha_- Q \be_+)(\de-\ga_- Q \be_+)^{-1}\gamma_-=\alpha_-^\circ,
\end{align*}
and it follows that $M_Q$ is invertible if and only if $\alpha_-^\circ$ in \eqref{FKRalcirc}, for this choice of $Q$, is invertible. However, invertibility of $M_Q$ is also equivalent to invertibility of $R(0)$, and hence independent of $Q$. Therefore, $\alpha_-^\circ$ in \eqref{FKRalcirc} is either invertible for all $Q\in\cB(\cX_+,\cX_-)$ such that $\de-\ga_- Q \be_+$ is invertible, precisely when $R(0)$ is invertible, or invertible for no $Q\in\cB(\cX_+,\cX_-)$ such that $\de-\ga_- Q \be_+$ is invertible, precisely when $R(0)$ is not invertible. The formula for $(\alpha_-^\circ)^{-1}$ follows from the standard Schur complement inverse formula, cf., pages 28--29 in \cite{BGKR08}.   
\end{proof}

The proof above makes extensive use of Schur complements. The following formulas for $M_Q$ may also have been used to prove part of the lemma. One easily checks that 
\[
M_Q=\begin{bmatrix} \alpha_- & 0 \\ \gamma_- & I\end{bmatrix}
\begin{bmatrix} I & \alpha_-^{-1}\beta_- - Q\beta_+ \\ 0 & \delta -\gamma_-\alpha_-^{-1}\beta_-\end{bmatrix}, \qquad 
M_Q=\begin{bmatrix} \alpha_- & \beta_-  \\ \gamma_- & \delta \end{bmatrix}
\begin{bmatrix} I & -Q\beta_+ \\ 0 & I \end{bmatrix}.
\]

Negating the implications in Lemma \ref{L:R(0)inv}, we obtain the following implications: 

\begin{corollary}\label{Cor:alpha-circ}
Let $R(z)$ be given by the stable representation \eqref{FKRreal} with $\alpha_-$ invertible. Then:
\begin{enumerate}
\item If there exists a $Q\in\cB(\cX_+,\cX_-)$ such that $\delta-\gamma_-Q\beta_+$ is invertible but $\alpha_-^\circ$ in \eqref{FKRalcirc} is not invertible, then $R(0)=\delta-\gamma_-\alpha_-^{-1}\beta_-$ is not invertible.

\item If $R(0)$ is not invertible, then for each $Q\in\cB(\cX_+,\cX_-)$ with $\de-\ga_- Q \be_+$ invertible, $\alpha_-^\circ$ in \eqref{FKRalcirc} is not invertible.
\end{enumerate}
\end{corollary}

%\begin{remark}
%The proof of Lemma \ref{L:R(0)inv} also provides an alternative formula for $R(0)^{-1}$, in the case where $\delta$ is not invertible: First note that $R(0)=\de-\ga_-\alpha_-^{-1}\beta_-$ is the Schur complement of $M_0=\sbm{\al_-&\be_-\\ \ga_- & \de}$ with respect to $\al_-$. If $\de$ is also invertible, then the Schur complement $M_0/\de=\al_--\be_- \de^{-1}\ga_-$ is Schur coupled with $R(0)$ and $R(0)$ is invertible if and only if $M_0/\de$ is invertible, with $R(0)^{-1}=\de^{-1}+\de^{-1}\ga_-(\al_--\be_- \de^{-1}\ga_-)^{-1}\be_- \de^{-1}$; again see pages 28--29 in \cite{BGKR08} for details. If $\de$ is not invertible, but we have a $Q\in\cB(\cX_+,\cX_-)$ such that $\de-\ga_- Q \be_+$ is invertible, then we still find that 
%\begin{align*}
%&R(0)^{-1}=(\de-\ga_- Q \be_+)^{-1}+(\de-\ga_- Q \be_+)^{-1}\ga_- \times\\
%&\quad \times (\al_--(\be_- -\al_- Q \be_+)(\de-\ga_- Q \be_+)^{-1}\ga_-)^{-1}(\be_- -\al_- Q \be_+) (\de-\ga_- Q \be_+)^{-1}
%\end{align*}
%and this formula holds for any $Q\in\cB(\cX_+,\cX_-)$ with $\de-\ga_- Q \be_+$ invertible, and is %independent of the choice of $Q$.  
%\end{remark}

The following example will show that $M_Q$ can be invertible without there existing a $Q$ such that $\delta -\gamma_- Q \beta_+$ is invertible. In fact, $R(0)$ is invertible, and hence $M_Q$, but $\delta-\gamma_- Q \beta_+$ is not invertible for any $Q$. In this case \eqref{RiccFKR} cannot have solutions, but \eqref{RiccFKR2} can.

\begin{example}\label{E:RiccEqCounter2}
Take 
\[
\delta:=\begin{bmatrix} 0 & 0 \\ 0 & 1 \end{bmatrix},\ \gamma_-:=\begin{bmatrix} 1 \\ 0 \end{bmatrix},\
\beta_+:=\begin{bmatrix} 0 & 1 \end{bmatrix} ,\ \beta_-:=\begin{bmatrix} 1 & 0 \end{bmatrix},\ \mbox{$\alpha_-, \alpha_+\in\BD$ with $\alpha_-\not= 0$.}
\]
We determine $\gamma_+$ at a later stage. For each $q\in\BC$ we have
\[
\delta-\gamma_- q \beta_+ =\begin{bmatrix} 0 & -q \\ 0 & 1\end{bmatrix},
\]
which is not invertible, while
\[
R(0)=\delta -\gamma_-\alpha_-^{-1}\beta_- =\begin{bmatrix} -\alpha_-^{-1} & 0 \\ 0 & 1 \end{bmatrix}
\]
is invertible. Moreover, the matrix $M_Q$ in the proof of Lemma \ref{L:R(0)inv} is given by 
\[
M_Q=\begin{bmatrix} \alpha_- & 1 & 0 \\ 1 & 0 &0 \\ 0 & 0 & 1 \end{bmatrix}
\begin{bmatrix} 1 & 0 & -q \\ 0 & 1 & 0 \\ 0 & 0 & 1\end{bmatrix},
\] 
which is invertible for every $q\in\BC$. However, since $\delta-\gamma_- q \beta_+$ is not invertible, $\al_-^\circ$ and $\al_+^\circ$ cannot be defined using \eqref{FKRalcirc}. Moreover, an easy computation shows that the right hand side of \eqref{alphacirc-inv} becomes 0, as it should be according to Lemma \ref{L:InvCon} below.  

The Riccati equation \eqref{RiccFKR2} becomes
\begin{align*}
q&=\alpha_-\alpha_+ q + \begin{bmatrix} 1 & -\alpha_- q\end{bmatrix}
\begin{bmatrix} -\alpha_- & 0 \\ 0 & 1\end{bmatrix} \left(\gamma_+ - \begin{bmatrix} \alpha_-^{-1}q \\ 0\end{bmatrix}   \right)\\
& =
\alpha_-\alpha_+ q +q + \begin{bmatrix} 1 & -\alpha_- q\end{bmatrix}
\begin{bmatrix} -\alpha_- & 0 \\ 0 & 1\end{bmatrix} \gamma_+ .
\end{align*}
Taking $\gamma_+:=\sbm{1\\0}$ will yield $q=\tfrac{1}{\alpha_+}$ as the unique solution. On the other hand, taking $\gamma_+:=\sbm{0\\\al_+}$, equation \eqref{RiccFKR2} becomes trivial, so that every $q\in\BC$ is a solution. Nonetheless, equation \eqref{RiccFKR2} does not have a stabilizing solution, since \eqref{StabCond+} in particular requires the right hand side of \eqref{alphacirc-inv} to be anti-stable; see also \eqref{FKRtoDichot}.\hfill$\Box$
\end{example} 

Even if $R(0)$ is invertible, and hence $M_Q$ in the proof of Lemma \ref{L:R(0)inv}, it can happen that $Q$ is such that $\de-\ga_- Q \be_+$ is not invertible, and hence $\alpha_-^\circ$ in \eqref{FKRalcirc} is not defined. In the next lemma we characterize when $\de-\ga_- Q \be_+$ is invertible, when we know in advance that $R(0)$ is invertible.  

\begin{lemma}\label{L:InvCon}
Let $R(z)$  be given by the stable representation \eqref{FKRreal} with $\alpha_-$ invertible. Assume that $R(0)=\de-\gamma_- \alpha_-^{-1}\beta_-$ is invertible. Let $Q\in\cB(\cX_+,\cX_-)$. Then $\de- \gamma_- Q \be_+$ is invertible if and only if 
the right-hand side of \eqref{alphacirc-inv} is invertible. In that case, $\alpha_-^{\circ}$ in \eqref{FKRalcirc} is invertible with $(\alpha_-^{\circ})^{-1}$ given by \eqref{alphacirc-inv}.
\end{lemma}

\begin{proof}[\bf Proof]
For $Q\in\cB(\cX_+,\cX_-)$ we now define 
\[
N_Q:=\mat{\alpha_-& \beta_--\alpha_- Q \be_+\\ -\gamma_- & \de-\gamma_- \alpha_-^{-1}\beta_-}=\mat{\alpha_-& \beta_--\alpha_- Q \be_+\\ -\gamma_- & R(0)}.
\]
Then the Schur complement of $N_Q$ with respect to $\alpha_-$ is
\[
N_Q/\alpha_- = \de-\gamma_- \alpha_-^{-1}\beta_- + \gamma_- \alpha_-^{-1} (\beta_--\alpha_- Q \be_+)= \de- \gamma_- Q \be_+,
\]
while the Schur complement of $N_Q$ with respect to $R(0)$ is given by 
\begin{align*}
N_Q/ R(0)
&= \alpha_- + (\beta_--\alpha_- Q \be_+) (\de-\gamma_- \alpha_-^{-1}\beta_-)^{-1} \gamma_-\\
&=\alpha_-(\alpha_-^{-1}+\alpha_-^{-1}(\beta_-- \alpha_- Q\be_+) (\de-\gamma_-\alpha_-^{-1}\beta_-)^{-1}\gamma_-\alpha_-^{-1})\alpha_-.
\end{align*}
We can then conclude that $N_Q$ is invertible if and only if $\de- \gamma_- Q \be_+$ is invertible if and only if $\alpha_-^{-1}+\alpha_-^{-1}(\beta_-- \alpha_- Q\be_+) (\de-\gamma_-\alpha_-^{-1}\beta_-)^{-1}\gamma_-\alpha_-^{-1}$ is invertible. In this case, since $R(0)$ is assumed to be invertible, we know from Lemma \ref{L:R(0)inv} that $\al_-^\circ$ in \eqref{FKRalcirc} is invertible with inverse given by \eqref{alphacirc-inv}.  
\end{proof}

Much of the above relies on the existence of a $Q\in\cB(\cX_+,\cX_-)$ such that $\de- \gamma_- Q \be_+$ is invertible. However, such a $Q$ need not exist. In the case that $\cU$ is finite dimensional, we characterize when this happens in the next general lemma.

\begin{lemma}\label{L:InvExt}
Let $U\in\cB(\BC^n)$, $V\in\cB(\cX_1,\BC^n)$ and $W\in\cB(\BC^n,\cX_2)$, with $\cX_1$ and $\cX_2$ possibly infinite dimensional Hilbert spaces. Then there exists a $Q\in\cB(\cX_2,\cX_1)$ such that $U-VQW$ is invertible if and only if $\im P_{(\im V)^\perp} U= (\im V)^\perp$ and $\kr U|_{\kr W}=\{0\}$. 
\end{lemma}

\begin{proof}
Set $\cV:=\im V$ and $\cW:=(\kr W)^\perp$. Assume that $Q\in\cB(\cX_1,\cX_1)$ is such that $U-VQW$ is invertible. Then necessarily $\cV^\perp =\im P_{\cV^\perp}(U-VQW)=\im P_{\cV^\perp} U$ and $\{0\}=\kr (U-VQW)|_{\cW^\perp}=\kr U|_{\cW^\perp}$. Hence the conditions $\im P_{(\im V)^\perp} U= (\im V)^\perp$ and $\kr U|_{\kr W}=\{0\}$ are necessary.  

For the converse claim, assume that $\im P_{(\im V)^\perp} U= (\im V)^\perp$ and $\kr U|_{\kr W}=\{0\}$. Decompose $\BC^n=\cV^\perp \oplus \cV$ and $\BC= \cW^\perp \oplus \cW$, and write out the operators with respect to this decomposition:
\[
U=\mat{U_{11}&U_{12}\\U_{21}&U_{22}},\quad V=\mat{0\\V_2},\quad W=\mat{0 & W_2}. 
\]
By definition if $\cV$ and $\cW$, then $\im V_2=\cV$ and $\kr W_2=\{0\}$. Since $\cV$ and $\cW$ are finite dimensional, it the follows that $V_2$ has a right inverse $V_2^+$ and $W_2$ has a left inverse $W_2^+$. It then follows that for any $T\in\cB(\cW,\cV)$, if we take $Q:=V_2^+ (U_{22}-T) W_2^+$, then $U_{22}-V_2 Q W_2=T$, and hence 
\[
U-VQW = \mat{U_{11}&U_{12}\\U_{21}&U_{22}- V_2 Q W_2}=\mat{U_{11}&U_{12}\\U_{21}&T}.
\]
The question thus becomes whether we can replace the $(2,2)$ entry in $U$ in such a way that the resulting matrix in $\BC^{n \times n}$ becomes invertible. Now further decompose
\begin{align*}
\cV^\perp = (\kr U_{11})^\perp \oplus \kr U_{11},\ \ & \ \ \cW_\perp =\im U_{11} \oplus (\im U_{11})^\perp, \\
\cV= \im U_{21}|_{\kr U_{11}} \oplus (\im U_{21}|_{\kr U_{11}})^\perp, \ & \ 
\cW= (\kr P_{\im U_{11}^\perp} U_{12})^\perp \oplus \kr P_{\im U_{11}^\perp} U_{12},
\end{align*}
and note that with respect to this further decomposition we have 
\[
U-VQW = \mat{\wtilU_{11}&0&U'_{12}&U''_{12}\\0&0&\wtilU_{12}&0\\U'_{21}&\wtilU_{21}&T_{11}&T_{12}\\U''_{21}&0&T_{21}&T_{22}}
\]
with $\wtilU_{11}$ invertible, and with $\wtilU_{12}$ and $\wtilU_{21}$ invertible by our assumption. Then $U-VQW$ is invertible if and only if the Schur complement of $U-VQW$ with respect to $\wtilU_{11}$ is invertible, and this Schur complement works out as
\[
U-VQW/\wtilU_{11}:=\mat{0&\wtilU_{12}&0\\\wtilU_{21}&T_{11}&T_{12}\\0&T_{21}&T_{22}}
-\mat{0\\U'_{21}\\U''_{21}}\wtilU_{11}^{-1}\mat{0&U'_{12}&U''_{12}}.
\]  
If we now take 
\[
T:=\mat{U'_{21}\\U''_{21}}\wtilU_{11}^{-1}\mat{U'_{12}&U''_{12}}+\mat{0&0\\0&T_0} 
\]
with $T_0\in\cB(\kr P_{\im U_{11}^\perp} U_{12},(\im U_{21}|_{\kr U_{11}})^\perp)$ invertible, then 
\[
U-VQW/\wtilU_{11}=\mat{0&\wtilU_{12}&0\\\wtilU_{21}&0&0\\0&0&T_0}
\]
is invertible. That one can take $T_0$ invertible, follows since $\dim \kr P_{\im U_{11}^\perp} U_{12}=\dim (\im U_{21}|_{\kr U_{11}})^\perp$, which in turn follows from a dimension count. 
\end{proof}

Next we consider what can happen to the solution sets of the two Riccati equations. For this purpose, given a stable representation \eqref{FKRreal} with $\alpha_-$ invertible, we define the following sets:
 \begin{align*}
\tu{Ricc}_1&:=\{Q\in\cB(\cX_+,\cX_-) \colon \mbox{$Q$ solves \eqref{RiccFKR}}\},\\
\tu{Ricc}_2&:=\{Q\in\cB(\cX_+,\cX_-) \colon \mbox{$Q$ solves \eqref{RiccFKR2}}\},\\
\tu{Ricc}_2'&:=\{Q\in\cB(\cX_+,\cX_-) \colon \mbox{$Q$ solves} \mbox{ \eqref{RiccFKR2} and $\de-\ga_- Q \be_+$ is invertible}\}. 
\end{align*}

\begin{proposition}\label{P:SolSetsRels}
Let $R(z)$  be given by the stable representation \eqref{FKRreal} with $\alpha_-$ invertible. If $R(0)$ is not invertible, then $\tu{Ricc}_2$ and $\tu{Ricc}_2'$ are empty, while $\tu{Ricc}_1$ might not be empty. 
If $R(0)$ is invertible, then $\tu{Ricc}_1=\tu{Ricc}_2' \subset \tu{Ricc}_2$, and the latter inclusion might be strict. 
\end{proposition}

\begin{proof}[\bf Proof]
If $R(0)$ is not invertible, then the Riccati equation \eqref{RiccFKR2} can never be solved, hence $\tu{Ricc}_2$ and $\tu{Ricc}_2'$ are empty. That $\tu{Ricc}_1$ need not be empty in this case, follows from Example \ref{E:RiccEqCounter}. Now assume that $R(0)$ is invertible. The inclusion $\tu{Ricc}_2' \subset \tu{Ricc}_2$ is obvious, and Example \ref{E:RiccEqCounter} again shows that it can be strict.

To complete the proof, it remains to show that whenever $Q\in\cB(\cX_+,\cX_-)$ is such that $\de-\ga_- Q \be_+$ is invertible, then $Q$ is a solution to \eqref{RiccFKR} if and only if it is a solution to \eqref{RiccFKR2}. To show that this is the case, we follow the argument from \cite{LRR86}, where a similar result is proved in the Hermitian case.

Let $Q$ be a solution to \eqref{RiccFKR}. Then
\begin{align*}
\gamma_- Q \al_+
& = \gamma_- \alpha_-^{-1} \alpha_- Q \al_+\\
&= \gamma_- \alpha_-^{-1} Q  - \gamma_- \alpha_-^{-1} (\beta_--\alpha_- Q\be_+)(\de-\gamma_- Q\be_+)^{-1}(\ga_+-\gamma_- Q\al_+)\\
&=\gamma_- \alpha_-^{-1} Q  - (\gamma_- \alpha_-^{-1}\beta_--\gamma_- Q\be_+)(\de-\gamma_- Q\be_+)^{-1}(\ga_+-\gamma_- Q\al_+).
\end{align*}
Hence, we have that
\begin{align*}
&\ga_+-\gamma_- \alpha_-^{-1} Q=\\
&\qquad = \ga_+ - \gamma_- Q \al_+ - (\gamma_- \alpha_-^{-1}\beta_--\gamma_- Q\be_+)(\de-\gamma_- Q\be_+)^{-1}(\ga_+-\gamma_- Q\al_+)\\
&\qquad =(\de-\gamma_- Q\be_+ - \gamma_- \alpha_-^{-1}\beta_-+\gamma_- Q\be_+)(\de-\gamma_- Q\be_+)^{-1}(\ga_+-\gamma_- Q\al_+)\\
&\qquad =(\de - \gamma_- \alpha_-^{-1}\beta_-)(\de-\gamma_- Q\be_+)^{-1}(\ga_+-\gamma_- Q\al_+).
\end{align*}
Thus, since $\de - \gamma_- \alpha_-^{-1}\beta_-=R(0)$ is invertible, we see that
\begin{equation}\label{RiccToQuad}
(\de-\gamma_- Q\be_+)^{-1}(\ga_+-\gamma_- Q\al_+)=(\de - \gamma_- \alpha_-^{-1}\beta_-)^{-1}(\ga_+-\gamma_- \alpha_-^{-1} Q).
\end{equation}
From this identity it is immediately clear that $Q$ solves \eqref{RiccFKR2}, since it solves \eqref{RiccFKR}.

Conversely, assume that $Q$ solves \eqref{RiccFKR2}. Again, it suffices to show that \eqref{RiccToQuad} holds. In this case we have
\begin{align*}
&\gamma_- \alpha_-^{-1}Q =\\ 
&\qquad = \gamma_- \alpha_-^{-1} \alpha_- Q \al_+ + \gamma_- \alpha_-^{-1}(\beta_--\alpha_- Q\be_+) (\de-\gamma_-\alpha_-^{-1}\beta_-)^{-1}(\ga_+-\gamma_-\alpha_-^{-1} Q)\\
&\qquad =\gamma_- Q \al_+ + (\gamma_- \alpha_-^{-1}\beta_- - \gamma_- Q \be_+) (\de-\gamma_-\alpha_-^{-1}\beta_-)^{-1}(\ga_+-\gamma_-\alpha_-^{-1} Q),
\end{align*}
so that
\begin{align*}
&\ga_+ - \gamma_- Q \al_+=\\  
&\qquad = \ga_+ - \gamma_- \alpha_-^{-1}Q + (\gamma_- \alpha_-^{-1}\beta_- - \gamma_- Q \be_+) (\de-\gamma_-\alpha_-^{-1}\beta_-)^{-1}(\ga_+-\gamma_-\alpha_-^{-1} Q)\\
&\qquad = (\de-\gamma_-\alpha_-^{-1}\beta_- + \gamma_- \alpha_-^{-1}\beta_- - \gamma_- Q \be_+) (\de-\gamma_-\alpha_-^{-1}\beta_-)^{-1}(\ga_+-\gamma_-\alpha_-^{-1} Q)\\
&\qquad = (\de- \gamma_- Q \be_+) (\de-\gamma_-\alpha_-^{-1}\beta_-)^{-1}(\ga_+-\gamma_-\alpha_-^{-1} Q),
\end{align*}
again leading to \eqref{RiccToQuad}, since $\de- \gamma_- Q \be_+$ is invertible.
\end{proof}

\begin{corollary}\label{C:al+circform}
Let $R(z)$ be given by the stable representation \eqref{FKRreal} with $\alpha_-$ and $R(0)$ invertible. If $Q\in\cB(\cX_+,\cX_-)$ solves \eqref{RiccFKR}, then $\al_+^\circ$ in \eqref{FKRalcirc} is also given by 
\begin{equation}\label{Acirc2-FKR}
\al_+^\circ =\al_+-\be_+(\de-\gamma_-\alpha_-^{-1}\beta_-)^{-1}(\ga_+-\gamma_-\alpha_-^{-1} Q).
\end{equation}
In this case, the (unique) stabilizing solutions in $\tu{Ricc}_1$ and $\tu{Ricc}_2'$ coincide. 
\end{corollary}

\begin{proof}[\bf Proof]
It follows from the proof of Proposition \ref{P:SolSetsRels} that \eqref{RiccToQuad} holds whenever $Q\in \tu{Ricc}_1\cup\tu{Ricc}_2'$. From \eqref{RiccToQuad}, it in turn follows immediately that $\al_+^\circ$ in \eqref{FKRalcirc} equals \eqref{Acirc2-FKR}, the first operator in \eqref{RiccFKR2stableOps}. 
%
%To show that the stabilizing solutions coincide we need to show that for an arbitrary $Q\in \tu{Ricc}_1\cup\tu{Ricc}_2'$, the operators in \eqref{FKRalcirc} are exponentially stable if and only if the inclusions \eqref{StabCond+} hold. To see this, 
%note that the first operator in \eqref{StabCond+} can be written as
%\begin{equation}\label{eq:Acirc2-FKR2}
%	\alpha_+ -\beta_+(\delta -\gamma_-\alpha_-^{-1}\beta_-)^{-1}(\gamma_+ -\gamma_-\alpha_-^{-1}\whatQ)
%\end{equation}
%using \eqref{FKRtoDichot}. With $\what{Q}=Q$, we see from \eqref{Acirc2-FKR} that \eqref{eq:Acirc2-FKR2} is $\alpha_+^\circ$, and hence this stabilization condition coincides for the two Riccati equations. 
%
%The second operator in \eqref{StabCond+} can similarly be written as 
%$$
%	\alpha_-^{-1}+(\alpha_-^{-1}\beta_- - \whatQ\beta_+)(\delta -\gamma_-\alpha_-^{-1}\beta_-)^{-1}\gamma_-\alpha_-^{-1},
%$$
%and 
By \eqref{alphacirc-inv}, the second operator in \eqref{RiccFKR2stableOps} is equal to $(\alpha_-^{\circ})^{-1}$. %It follows that the second inclusion in \eqref{StabCond+} is $\sigma((\alpha_-^{\circ})^{-1}) \subset \BC\backslash \ov{\BD}$, which is equivalent to $\sigma(\alpha_-^\circ)\subset\BD\setminus\zero$. 
This shows that the stabilizing solutions in $\tu{Ricc}_1$ and $\tu{Ricc}_2'$ coincide, and we already know from Theorem \ref{thm:FKR10} that any stabilizing solutions in $\tu{Ricc}_1$ are unique. 
\end{proof}

Note that the uniqueness of the stabilizing solution for both Riccati equations is well-known; see, e.g., \cite{LaRoBook}.

\section{Left versus right canonical Wiener-Hopf factorization}\label{S:LeftRight}

In this section our starting point is a dichotomous realization \eqref{FKRtoDichot}. We shall also assume that $R(0)=\delta-\gamma_-\alpha_-^{-1}\beta_-$ is invertible. Consider
\begin{align*}
A^\times &= \begin{bmatrix}
\alpha_-^{-1} & 0 \\ 0 & \alpha_+
\end{bmatrix} -
\begin{bmatrix} \alpha_-^{-1} \beta_- \\ \beta_+\end{bmatrix} (\delta -\gamma_-\alpha_-^{-1}\beta_-)^{-1} 
\begin{bmatrix} -\gamma_-\alpha_-^{-1} & \gamma_+\end{bmatrix} \\
&=
\begin{bmatrix} 
\alpha_-^{-1} +\alpha_-^{-1} \beta_- R(0)^{-1} \gamma_-\alpha_-^{-1} & - \alpha_-^{-1} \beta_- R(0)^{-1} \gamma_+ \\
\beta_+R(0)^{-1}\gamma_-\alpha_-^{-1} & \alpha_+ -\beta_+R(0)^{-1}\gamma_+
\end{bmatrix}.
\end{align*}
We assume that $R(z)$ has a right canonical Wiener-Hopf factorization, so that there is a unique stabilizing solution $Q$ to \eqref{RiccFKR} by Theorem \ref{thm:FKR10}. The same $Q$ is the unique stabilizing solution to \eqref{RiccDichot2} by Corollary \ref{C:al+circform}. Then $Q$ is also the angular operator for the space $\cX_+^\times$ in the sense that the spectral subspace of $A^\times$ corresponding to $\BD$ is equal to ${\rm Im\,} \sbm{Q\\I}$; indeed, by Corollary \ref{C:al+circform} and \eqref{RiccDichot2}, we have
\[
A^\times \begin{bmatrix} Q \\ I \end{bmatrix} = \begin{bmatrix} Q \\ I \end{bmatrix} 
\big(\alpha_+-\beta_+R(0)^{-1} (\gamma_+-\gamma_-\alpha_-^{-1} Q)\big) =
\begin{bmatrix} Q \\ I \end{bmatrix}\alpha_+^\circ .
\]
Further, by \eqref{alphacirc-inv} we have
\[
A^\times \begin{bmatrix} I & Q \\  0 & I \end{bmatrix}  =
\begin{bmatrix} I & Q \\  0 & I \end{bmatrix}
\begin{bmatrix} (\alpha_-^\circ)^{-1} & 0 \\
\beta_+R(0)^{-1}\gamma_-\alpha_-^{-1} & \alpha_+^\circ\end{bmatrix}.
\]
To find the spectral subspace $\cX_-^\times$ of $A^\times$ corresponding to $\BC\backslash\ov{\BD}$ we block-diagonalize $A^\times$ as follows. Let $Z$ be the solution of the Lyapunov equation
\begin{equation}\label{LyapLR}
Z (\alpha_-^\circ)^{-1} -\alpha_+^\circ Z=-\beta_+R(0)^{-1}\gamma_-\alpha_-^{-1} ,
\end{equation}
which exists and is unique since $\si((\al_-^\circ)^{-1})\cap \si(\al_+^\circ)=\emptyset$; see, e.g., Theorem I.4.1 in \cite{GGKbook90}.
Then 
\[
\begin{bmatrix} I & 0 \\ Z & I \end{bmatrix} \begin{bmatrix} I & -Q \\ 0 & I \end{bmatrix} A^\times 
\begin{bmatrix}I & Q \\ 0 & I \end{bmatrix} \begin{bmatrix} I & 0 \\ -Z & I \end{bmatrix}  =
\begin{bmatrix} (\alpha_-^\circ)^{-1} & 0 \\
0 & \alpha_+^\circ\end{bmatrix} .
\]
So, the spectral subspace $\cX_-^\times$ of $A^\times$ corresponding to $\BC\backslash \ov{\BD}$ is given by 
\[
\cX_-^\times ={\rm Im\,} \begin{bmatrix} I-QZ \\ -Z\end{bmatrix}.
\]

This can now be applied to determine when there will be a left canonical Wiener-Hopf factorization of $R(z)$. Indeed, this will be the case when $\cX_+ \dot+ \cX_-^\times=\cX$, see, e.g., Theorem 6.1 in \cite{BGKR08}. Since $\cX_+={\rm Im\, } \begin{bmatrix} 0 \\ I \end{bmatrix}$ we see that this is the case precisely when $I-QZ$ is invertible, and in that case
\[
\cX_-^\times ={\rm Im\,} \begin{bmatrix} I \\ -Z(I-QZ)^{-1}\end{bmatrix}.
\]
We arrive at the following result:

\begin{theorem}\label{T:leftvsright}
Let $R(z)$ be given by \eqref{FKRreal}, with $R(z)$ analytic and invertible at zero, and assume that $R(z)$ admits right canonical Wiener-Hopf factorization. Let $Q$ be the stabilizing solution to \eqref{RiccFKR}. Also let $Z$ be the unique solution of \eqref{LyapLR}. Then $R(z)$ admits left canonical Wiener-Hopf factorization if and only if $I-QZ$ is invertible.
\end{theorem}

Compare Theorem \ref{T:leftvsright} to \cite{BallRan} and Theorem 12.6 in \cite{BGKR10}, where the same question was studied, but from a different starting point.

\subsection*{Funding}

This work is based on research supported in part by the National Research Foundation of South Africa (Grant Numbers 127364 and 145688) and the DSI-NRF Centre of Excellence in Mathematical and Statistical Sciences (CoE-MaSS). Any opinion, finding and conclusion or recommendation expressed in this material is that of the authors and the NRF and CoE-MaSS do not accept any liability in this regard. Part of the research was conducted during a visit of the second author to North-West University in March and April of 2024 supported by a scholarship of the Magnus Ehrnrooth Foundation.

% MK: Leave these in, please!
%\bibliography{cont}
%\bibliographystyle{amsalpha}
%\end{document}

\end{document}